\documentclass[a4paper,11pt]{article}

\usepackage[arxiv]{optional}

\opt{arxiv}{ 
  \usepackage{amsmath,amsthm}
}
\usepackage{enumerate}

\usepackage[T1]{fontenc}
\usepackage[utf8]{inputenc}
\usepackage{microtype}
\opt{arxiv}{
  \usepackage{amssymb}
  
  \IfFileExists{newtxmath.sty}%
    {\usepackage{libertine}
    \usepackage[libertine]{newtxmath}
    \useosf 
    }%
    {\usepackage[osf]{libertine}}
  \IfFileExists{zi4.sty}%
    {\usepackage[scaled=0.96]{zi4}}  
    {\usepackage[scaled=0.83]{beramono}}
}

\usepackage[dvipsnames]{xcolor}
\definecolor{darkgreen}{rgb}{0,0.45,0}
\definecolor{darkred}{rgb}{0.65,0,0}
\definecolor{darkblue}{rgb}{0,0,0.6}
\opt{arxiv}{
  \usepackage[colorlinks,citecolor=darkgreen,linkcolor=darkred,urlcolor=darkblue]{hyperref}
}
\usepackage{zref-clever}
\zcsetup{cap,nameinlink}

\opt{arxiv}{
  \usepackage{fancyhdr}

  \fancypagestyle{firstpage}{
  \fancyhf{}
  \fancyhead[R]{\arxivnote}
  \fancyhead[L]{\doinote}
  \fancyfoot[C]{\thepage}
}

}

\usepackage{needspace}
\usepackage{mathpartir}
\usepackage{xifthen} 

\opt{arxiv}{
  \usepackage{geometry}
  \geometry{b5paper,textwidth=360pt,textheight=595pt,headsep=1ex,footskip=4ex} 
}

\NewDocumentCommand{\newzctheorem}{mmO{#1}}{
  \newtheorem{#1}[sharedtheoremcounter]{#2}
    \AddToHook{env/#1/begin}{%
      \zcsetup{countertype={sharedtheoremcounter=#3}}}
}

\theoremstyle{plain}
\newzctheorem{theorem}{Theorem}
\newzctheorem{remark}{Remark}
\newzctheorem{lemma}{Lemma}
\newzctheorem{proposition}{Proposition}
\newzctheorem{corollary}{Corollary}
\newzctheorem{scholium}{Scholium}
\theoremstyle{definition}
\newzctheorem{definition}{Definition}

\input{non-discrete-macros.tex}

\begin{document}

\opt{jsl}{
  \title[Constructive reflectivity principles for regular theories]%
  {Constructive reflectivity principles \\ for regular theories}
}
\opt{arxiv}{
  \title{Constructive reflectivity principles \\ for regular theories%
  \footnote{An earlier version was entitled ``\ldots reflection principles\ldots'';
  we have amended the title, at Colin McLarty's suggestion, to avoid clashing with the classical sense of that phrase.}}
}

\newcommand{\ourthanks}{This research was funded in part by the Swedish Research Council grant no.\ 2008-05076, and the Research Council of Norway grant no.\ 230525.}

\opt{arxiv}{
\author{Henrik Forssell \and Peter LeFanu Lumsdaine}
}
\opt{jsl}{
\author{Henrik Forssell}
\revauthor{Forssell, Henrik}
\twoaddress%
{Department of Informatics \\ University of Oslo \\ Postboks 1080 Blindern \\ 0316 Oslo \\ Norway}
{Department of Mathematics and Science education \\ University of South-Eastern Norway \\ Papirbredden 1 \\ 3045 Drammen \\ Norway}
{5}
\email{jonf@ifi.uio.no}
\thanks{\ourthanks}

\author{Peter LeFanu Lumsdaine}
\revauthor{Lumsdaine, Peter LeFanu}
\address{Department of Mathematics \\ Stockholm University \\ SE-106 91 Stockholm \\ Sweden }
\email{p.l.lumsdaine@math.su.se}
}

\opt{arxiv}{
  \date{April 13, 2016
    \\ \footnotesize (this revision April 13, 2026)}
}

\opt{jsl}{
  \date{\todo{look up original!}}
}

\opt{jsl}{
}

\newcommand{\doinote}{\doi{10.1017/jsl.2019.70}}
\newcommand{\arxivnote}{\arxiv{1604.03851}}

\opt{arxiv}{
  \titlenote{\doinote \hspace{1.5em} \arxivnote}
}

\opt{arxiv}{
  \maketitle
  \thispagestyle{firstpage}
}

\begin{abstract}
  Classically, any structure for a signature $\Sigma$ may be completed to a model of a desired regular theory $\theory$ by means of the \emph{chase construction} or \emph{small object argument}.
  Moreover, this exhibits $\Mod{\theory}$ as weakly reflective in $\Str{\Sigma}$.

  We investigate this in the constructive setting.
  The basic construction is unproblematic; however, it is no longer a weak reflection.
  Indeed, we show that various reflectivity principles for models of regular theories are equivalent to choice principles in the ambient set theory.
  However, the embedding of a structure into its chase-completion still satisfies a \emph{conservativity} property, which suffices for applications such as the completeness of regular logic with respect to Tarski (i.e.~set) models.

  Unlike most constructive developments of predicate logic, we do not assume that equality between symbols in the signature is decidable. 
  While in this setting, we also give a version of one classical lemma which is trivial over discrete signatures but more interesting here: the abstraction of constants in a proof to variables.
\end{abstract}

\opt{arxiv}{
}

\opt{jsl}{\maketitle}

\section{Introduction}\label{section:intro}

Most developments of first-order logic are given in a classical meta-theory.
Even presentations in constructive settings usually assume that the signature is \emph{discrete}: that is, that decidability (excluded middle) holds for equality of the basic function and predicate symbols.
This precludes various classical constructions, such as the diagram of a structure, or the internal language of a category.

In many cases, the restriction to discrete signatures can be dropped with little consequence. 
Many standard results and constructions, especially as developed in the categorical tradition, go through constructively for arbitrary signatures with no significant modification.
Some, however, do not adapt so straightforwardly.

Here we set down constructive versions of two such results, for arbitrary signatures: the ``chase'' construction for producing models of regular (or ``positive-primitive'') theories; and  the abstraction of constants in a proof to variables. 

\subtleparagraph Most substantially, we investigate the \emph{chase construction} for regular theories. 

This is one of a family of similar constructions which have been invented independently in several fields.
We draw the name \emph{chase} from database theory \cite[\textsection 8.4]{abiteboulhullvianu:95}; categorically, it is a form of the \emph{small object argument} \cite{adamek-herrlich-rosicky-tholen}.

In each case, the idea is to construct a model of a regular theory\footnote{Or, in categorical terms, to obtain an injectivity condition.}, starting from a given structure, by iteratively adjoining elements to witness all existential axioms, and updating the basic predicates as necessary.
Classically, this provides a weak reflection from arbitrary structures into models of the theory.
That is, it provides for each structure $A$ a model $\chase{A}$ and homomorphism $\eta : A \to<125> \chase{A}$, such that any homomorphism  $f : A \to<125> M$  from $A$ to a model extends along $\eta$ to a homomorphism $\bar{f} : \chase{A} \to<125> M$.

In \zcref{sec:conservativity} we give a functorial chase construction for regular theories, and show that this provides a \emph{\theory-conservative} map from any structure into a model of the theory.
This suffices for applications including the regular and geometric completeness of regular logic (with respect to Tarski models).

However, unless choice holds, this cannot provide a weak reflection as it does classically.
Indeed, in \zcref{section: reflectivity and choice principles}, we show (using variations of the chase) that this and related reflectivity principles for regular theories are equivalent to choice principles in the ambient set theory. 

\subtleparagraph Chase constructions can be regarded and used syntactically as a kind of proof calculus \cite{costelombardiroy:01}: presentation formulas take the place of structures, and results of the process correspond to provable consequences.
This exploits the interchangeability of variables and constants/elements, something classically too simple to usually warrant more than a throwaway remark:
any proof mentions only finitely many constants of the language, so they (or any subset) may be replaced by free variables (“abstracted away”) throughout.

Over non-discrete signatures, this is no longer trivial, since the set of constants occurring in the proof may not be discrete, nor hence constructively finite.
With a little more effort, however, a version of the result still holds, by looking at which \emph{occurrences} of the constants are required to be equal by the proof, and using this to perform the abstraction.
This is the main content of \zcref{sec:consts-by-vbles}, accompanied by a sample application over diagrams of models. 

\subtleparagraph Having recalled some background in \zcref{section:prelims}, we treat the abstraction of constants first in \zcref{sec:consts-by-vbles}, as a warm-up to the main results on the chase construction and reflectivity principles in \zcref{sec:conservativity,section: reflectivity and choice principles}.
The two topics are each self-contained, however, so the reader who wishes to skip \zcref{sec:consts-by-vbles} and cut to the chase may safely do so.

\subsection*{Acknowledgements}

For helpful discussions and correspondence on the material of this paper, we would like to thank (in approximately chronological order) Christian Espíndola, Erik Palmgren, Håkon Robbestad Gylterud, Ji\v{r}\'{i} Rosick\'{y}, Michael Rathjen, and Marek Zawadowski.

\opt{arxiv}{\ourthanks}

\section{Preliminaries}\label{section:prelims}

As mentioned in the introduction, constructive treatments of syntax (e.g.\ Troelstra, Van Dalen \cite[\textsection 1.1]{troelstravandalen:88}, Fitting \cite[\textsection 4.1]{fitting:69}, and Veldman \cite{veldman:76}) typically assume that the sets of basic function and relation symbols are discrete, and often moreover enumerable, though these assumptions are rarely used.
This is at odds with classical and categorical model theory, where it is standard to consider signatures consisting of arbitrary sets; and this extra generality is needed for instance in taking the diagram of a structure, where one adds the elements of the structure to the language as constants.

In fact, especially in categorical treatments (e.g.\ Johnstone \cite[D1.1]{elephant1}, Diaconescu \cite[\textsection 3.1]{diaconescu08}), many proofs and constructions over arbitrary signatures are already fully constructive, and this is generally well-known in the field.
However, we are unaware of any source that lays out the fundamentals of first-order model theory over arbitrary signatures in an explicitly constructive setting.
We therefore recall basic background results in slightly more detail than usual, to put on record the fact that they and their standard proofs are constructively unproblematic.

Precisely, we work throughout in the constructive set theory IZF, as given in Troelstra, Van Dalen \cite[\textsection 11.8]{troelstravandalen:88}.
\zcref[S]{sec:conservativity,sec:consts-by-vbles} may comfortably also be read as working over IHOL$\infty$ (the logic of an elementary topos with a natural numbers object), set out in Lambek, Scott \cite{lambek-scott:1986}, or over extensional type theory with an impredicative universe of propositions, in Hofmann \cite{hofmann:1995}.
\zcref[S]{section: reflectivity and choice principles}, however, deals essentially with issues of unbounded quantification, which cannot be so straightforwardly translated to those settings.

\subsection*{Syntax}

Throughout, we say a set $X$ is \emph{discrete} if it has decidable equality: for all $x,y \in X$, either $x = y$ or $x \neq y$.%
\footnote{This is distinct from the homotopical sense of “discrete” in Homotopy Type Theory.}
By \emph{finite}, we mean \emph{cardinal-finite}, i.e.\ bijectable with some set of the form $\{1,\ldots,n\}$.
A set is \emph{Kuratowski-finite} or \emph{K-finite} if it admits a surjection from some cardinal-finite set.
A set is finite just if it is K-finite and discrete.

A \emph{family} means by default a set-indexed family.

A (finitary) signature $\Sigma$ consists of sets $\Sigma_{\text{Fun}},\Sigma_{\text{Rel}}$ of function and relation symbols, with a natural-number arities function $\arity{(-)} : \Sigma_{\text{Fun}} + \Sigma_{\text{Rel}} \to \nat$.
(We consider only single-sorted signatures.)

Syntax over such a signature is now constructed as usual (specifically, we follow Johnstone \cite[D1.1.3]{elephant1}), using a fixed set $V$ of variables, assumed discrete and constructively infinite\footnote{I.e.\ with an operation giving, for any finite list from $V$, some element not in the list.}.
We view the syntax as stratified into sets $\Tm[\Sigma]{I}$, $\Form[\Sigma]{I}$, for $I \subseteq_{\mathrm{fin}} V$, of “terms (resp.\ formulas) in context $I$”, i.e.\ with $\FV(t), \FV(\varphi) \subseteq I$.

As in Johnstone \cite[D1.1.3]{elephant1}, we consider fragments of logic including all atomic formulas but varying selections of connectives.
\emph{Horn} logic has just $\land$ and $\ltrue$\footnote{Classical Horn clauses correspond to Horn \emph{sequents} in this sense.};
\emph{regular} logic has these, plus $\exists$;
\emph{geometric} logic additionally has arbitrarily-indexed disjunctions, within finite contexts;
and \emph{first-order} logic has all usual (finitary) connectives and quantifiers.

While our contexts are formally finite subsets of $V$, we will often write them as lists $\x = x_1,\ldots,x_n$ of distinct variables, representing the set $\{x_1,\ldots,x_n\}$.
Similarly, our fundamental form of (simultaneous, capture-free) substitution is formally $\varphi[f]$, where $\varphi \in \Form{I}$ and $f : I \to \Tm{J}$; but for a finite context $\x$, we will often write this as $\varphi[\s/\x]$, where $s_i = f(x_i)$.
In the special case of a subset inclusion function $f : I' \subseteq I$, for $\varphi \in \Form{I'}$, we will write the weakening $\varphi[f] \in \Form{I}$ just as $\varphi$.

Finally, we will often display variables explicitly, introducing a formula as e.g.\ $\varphi(\x)$ to indicate $\varphi \in \Form{\x}$, and having done so, writing $\varphi(\s)$ for the substitution $\varphi[\s/\x]$.

A \emph{sequent} consists of a context and a pair of formulas in that context, written as the formal expression $\varphi \proves_I \psi$.
A sequent is \emph{regular}, \emph{geometric}, etc.\ if its formulas lie in that fragment.

A \emph{(regular, geometric, etc.)\ theory} $\theory$ is a set of sequents of the specified fragment, the \emph{axioms} of $\theory$.
We write $\varphi \proves^{\theory}_I \psi$ to indicate that the sequent $\varphi \proves_I \psi$ is derivable from axioms of $\theory$, in the fragment of logic under consideration.

A \emph{normal regular sequent} is one of the form $\varphi(\x) \proves_\x \fins{\y} \psi(\x,\y)$, with $\varphi$, $\psi$ Horn formulas, and such that $\psi(\x,\y) \proves^\emptyset_{\x,\y} \varphi(\x)$ is derivable.
A \emph{normal regular theory} is a family of normal regular sequents.
By Johnstone \cite[D1.3.10]{elephant1}, every regular sequent or theory can be canonically transformed to an equivalent normal one, which we call its \emph{normalisation}.

\subsection*{Semantics}

Structures for a signature, and the subsequent interpretation of first-order logic therein, are defined as usual.
We write $\Str{\Sigma}$ for the (complete and co-complete) category of $\Sigma$-structures, and $\Mod{\theory}$ for the full subcategory of $\theory$-models.

A map of signatures $f : \Sigma \to \Sigma'$ (i.e.\ functions $f_{\text{Fun}} : \Sigma_{\text{Fun}} \to \Sigma'_{\text{Fun}}$, $f_{\text{Rel}} : \Sigma_{\text{Rel}} \to \Sigma'_{\text{Rel}}$, preserving arities) induces on the one hand a translation of syntax $(-)^f : \Form[\Sigma]{} \to \Form[\Sigma']{}$, and on the other hand a reduct functor $f^* : \Str{\Sigma'} \to \Str{\Sigma}$.
These satisfy $\interp[f^*A]{ \x\ |\ \varphi } = \interp[A]{ \x\ |\ \varphi^f}$; so for any theory $\theory$, $f^*$ restricts to a functor $f^* : \Mod{\theory^f} \to \Mod{\theory}$.

For any Horn formula $\varphi(\x)$ (i.e.\ conjunction of atomic formulas), there is a structure $\syntstr{\x}{\varphi}$ which \emph{represents} the interpretation of $\varphi(\x)$: there is a natural isomorphism $\interp[A]{ \x\ |\ \varphi} \iso \operatorname{Hom}_{\Str{\Sigma}}(\syntstr{\x}{\varphi},A)$.
In case the signature is \emph{relational} (i.e.\ $\Sigma_{\text{Fun}} = \emptyset$), the domain $|\syntstr{\x}{\varphi}|$ is finite.
Explicitly, the domain of $\syntstr{\x}{\varphi}$ is the quotient of $\x$ by the (decidable) equivalence relation generated by the equality statements of $\varphi$; that is, such that $x_i\sim x_j$ iff $\varphi \proves^{\emptyset}_{\x} x_i = x_j$. 
  For $R \in \Sigma$,
\[\sem{R} :=
   \cterm%
     {\pair{ [x_{i_{0}}], \ldots, [x_{i_{\arity{R}-1}}] } \in (\x/{\sim})^{\arity{R}}}%
     {\varphi \proves^{\emptyset}_{\x}R(x_{i_0},\ldots, x_{i_{\arity{R}-1}})}.\]
For any regular formula $\theta(\x)$, $\syntstr{\x}{\varphi} \believes \theta([x_0],\ldots,[x_{n-1}])$ just if $\varphi(\x) \proves^{\emptyset}_\x \theta(\x)$.

A structure validates a regular normal sequent $\varphi(\x) \proves_\x \fins{\y} \psi(\x,\y)$ just if it is \emph{injective} with respect to the canonical map $\syntstr{\x}{\varphi} \to \syntstr{\x,\y}{\psi}$ (cf.\ \cite[Ex.~5.e]{adamekandrosicky:94}).

\section{Replacing constants by variables} \label{sec:consts-by-vbles}

As a first warm-up with non-discrete signatures, we give a version of the classical technique of abstracting away constants in a derivation to variables,
together with a sample application.

Before giving the general construction, we take an example, both to convince the reader that the statement is not quite trivial, and to illustrate the procedure used in the proof.

Suppose we have a derivation, from some theory $\theory$, of $\top\proves R(c_1,c_2)$.
We want to abstract this to a derivation not mentioning $c_1,c_2$, but with some free variables $\y$ instead, from which the original derivation can be recovered by substitution.

One's knee-jerk reaction might be to replace $c_1$ by $y_1$ throughout, and $c_2$ by $y_2$.
Of course, this is wrong: it works only if $c_1 \neq c_2$.
Classically, one can salvage this approach by working by cases.
Either $c_1 = c_2$ or $c_1 \neq c_2$, and in each case one gets a bijection from $\{c_1,c_2\}$ to some finite set of variables, which one can abstract along.

Constructively, however, we cannot in general make this case distinction, and hence cannot find such a bijection.
To find an abstraction, we must look not just at the conclusion but at the entire derivation.

For instance, if the derivation obtains $\top\proves R(c_1,c_2)$ from an axiom $\top\proves_{ x_1,x_2} R(x_1,x_2)$, then we abstract it to a derivation of $\top\proves_{ y_1,y_2} R(y_1,y_2)$.
In case $c_1 = c_2$, our abstraction is more general than the original derivation; but in any case it is general enough that substituting $c_i$ for $y_i$ recovers the original.

If instead $\top\proves R(c_1,c_2)$ is obtained from an axiom $\top \proves_{ x_1} R(x_1,x_1)$, then this shows that in fact $c_1$ and $c_2$ must be equal.
So in this case we abstract to a derivation of $\top\proves R(y_1,y_1)$, and know that substituting $c_1$ for $y_1$ gives back the original.

The key point is that while “$c_1 = c_2$?” may not be decidable, the weaker question “does this derivation \emph{require} $c_1 = c_2$?” always is, essentially since our syntax and deduction system is finitary.
So we replace all \emph{occurrences} of constants by variables, distinct as far as possible, and equal just when the derivation requires them to be.

\begin{lemma}\label{lemma: replacing constants with variables}
  Let $\Sigma$ be a signature, $C$ a decidable subset of its constants, and $\theory$ a theory over $\Sigma \setminus C$.
  Suppose $\varphi(\x)$, $\psi(\x)$ are formulas over $\Sigma$ such that $\varphi(\x) \proves^\theory_\x \psi(\x)$.
  Then there exist formulas $\bar{\varphi}(\x,\y)$ and $\bar{\psi}(\x,\y)$ over $\Sigma \setminus C$, and a function $f : \y \to C$, such that $\bar{\varphi}[f]=\varphi$, $\bar{\psi}[f]=\psi$, and $\bar{\varphi}(\x,\y) \proves^\theory_{\x,\y} \bar{\psi}(\x,\y)$.

  Moreover, the derivation of $\bar{\varphi}(\x,\y) \proves^\theory_{\x,\y} \bar{\psi}(\x,\y)$ may be taken to use the same logical rules and axiom schemes as the original derivation of $\varphi(\x) \proves^\theory_{\x} \psi(\x)$; in particular, to lie in the same fragment of logic.
\end{lemma}

\begin{proof}
  Recall that we work formally with the rules set out in Johnstone~\cite[D1.3.1]{elephant1}, i.e.~the standard intuitionistic sequent rules, presented in terms of sequents with a single antecedent.

  We work directly by induction over the form of the derivation of $\varphi$.
  
  \case{Axioms of $\theory$.}
  If the derivation consists of just an axiom $\varphi \proves_\x \psi$ of $\theory$, then since $\theory$ does not mention $C$, we are done by taking $\y = \emptyset$, $\bar{\varphi} = \varphi$, $\bar{\psi} = \psi$.

  \case{Structural and logical axioms.}
  If the derivation is an $\land$-elimination axiom
  \begin{mathpar}
    \inferrule{ }{\varphi_0 \land \varphi_1 \proves_\x \varphi_0}
  \end{mathpar}
  we take $\y$ to be fresh variables (i.e.\ distinct from $\x$) corresponding to the (finite) set of occurrences of constants from $C$ in either $\varphi_0$ or $\varphi_1$, and $f : \y \to C$ the function sending each variable to the constant appearing at the corresponding occurrence.
  Then taking $\bar{\varphi}_i$ to be $\varphi_i$ with each such occurrence replaced by the corresponding variable, and $\bar{\varphi} = \bar{\varphi}_0 \land \bar{\varphi}_1$, $\bar{\psi} = \bar{\varphi}_0$, we are done.

  Other axioms are entirely analogous.

  \case{Substitution rule.}
  If the derivation concludes with a substitution
\begin{mathpar}
\inferrule{\varphi' \proves_{\x'} \psi'}{\varphi'[\s/\x'] \proves_{\x} \psi'[\s/\x']}
\end{mathpar}
where $\s : \x' \to \Tm{\x}$, then by induction we have $\y'$, $f' : \y' \to C$, and $\bar{\varphi}' \proves_{\x',\y'} \bar{\psi}'$ as in the original statement.

  Choose fresh variables $\y''$ corresponding to occurrences of constants from $C$ in terms in $\s$, $f'' : \y'' \to C$ sending each variable to the corresponding constant, and $\bar{\s} : \x' \to \Tm{\x,\y''}$ to be $\s$ with occurrences replaced by the corresponding variables.

  Now taking $\y = (\y',\y'')$, $f = f' \cup f''$, $\bar{\varphi} = \bar{\varphi}'[\bar{\s}/\x']$, $\bar{\psi} = \bar{\psi}'[\bar{\s}/\x']$, we are done, concluding by the substitution
  \begin{mathpar}
    \inferrule{\bar{\varphi}' \proves_{\x',\y'} \bar{\psi}'}{\bar{\varphi}'[\bar{\s}/\x'] \proves_{\x,\y',\y''} \bar{\psi}'[\bar{\s}/\x']}.
  \end{mathpar}

  \case{Double rules.}
  The double rules for $\exists$, $\forall$, and $\Imp$ are straightforward, needing no modification of the $\y$, $f$ provided by the inductive hypothesis.

  \case{Multi-premise rules.}
  Suppose the derivation concludes with a cut
  \begin{mathpar}
    \inferrule{\varphi \proves_\x \chi \\ \chi \proves_\x \psi}{\varphi \vdash_{\x} \psi}.
  \end{mathpar}
  By induction we have abstractions $\y'$, $f'$, and $\bar{\varphi}' \proves^{\theory}_{\x,\y'} \bar{\chi}'$ of the first subderivation, and $\y''$, $f''$, and $\bar{\chi}'' \proves^{\theory}_{\x,\y''} \bar{\psi}''$ of the second.
  Without loss of generality, we may assume $\y'$, $\y''$ disjoint.

  Since $\bar{\chi}'[f'] = \bar{\chi}''[f''] = \chi$, occurrences of variables from $\y'$ in $\bar{\chi}'$ correspond to occurrences of variables from $\y''$ in $\bar{\chi}''$.
  Let $\sim$ be the equivalence relation on $\y',\y''$ generated by setting $y'_i \sim y''_j$ whenever some occurrence of $y'_i$ corresponds to some occurrence of $y''_j$

  As a finitely generated equivalence relation on a finite set, $\sim$ is decidable, so its quotient is finite, and can be represented by some fresh variables $\y$, with quotient map $q : \y',\y'' \to \y$.
  Now $f' \cup f''$ factors uniquely through $q$ as $f : \y \to C$, and $\chi'[q] = \chi''[q]$, so we are done, concluding
  \begin{mathpar}
    \inferrule
      {\inferrule{\bar{\varphi}' \proves_{\x,\y'} \bar{\chi}'}
        {\bar{\varphi}'[q] \proves_{\x,\y} \bar{\chi}'[q] } 
      \\ \inferrule{\bar{\chi}'' \proves_{\x,\y''} \bar{\psi}''}
        {\bar{\chi}''[q] \proves_{\x,\y} \bar{\psi}''[q] } }
      { \bar{\varphi}'[q] \proves_{\x,\y} \bar{\psi}''[q]  }.
  \end{mathpar}
The $\land$-introduction and $\lor$-elimination rules are analogous.
\end{proof}

We remark that this proof adapts directly to other standard forms of sequent calculus and natural deduction (including versions with proof-terms), and indeed to other finitary extensions of predicate logic (e.g.\ by modal operators and corresponding rules). 

\subtleparagraph A typical application of \zcref{lemma: replacing constants with variables} is in analysing provability over diagrams of models, where having added constants to the signature, one may wish to re-abstract them.

\begin{definition}
  Let $A$ be a $\Sigma$-structure.
  Then $\Sigma + \domain{A}$ is the signature obtained by adding constants to $\Sigma$ for all elements of $A$ (i.e.\ with $(\Sigma + \domain{A})_\Fun = \Sigma_\Fun + \domain{A}$), and $\Diag{A}$, the \emph{diagram} of $A$, is the theory over $\Sigma + \domain{A}$ consisting of the sequents $\ltrue \proves \varphi(\a)$, for each atomic predicate instance $\varphi(\a)$ that holds in $A$.
\end{definition}

\begin{proposition}
  For any regular formula $\varphi(\x)$ and $\a \in A^\x$, $A \believes \varphi(\a)$ if and only if $\ltrue \proves^{\Diag{A}} \varphi(\a)$.
\end{proposition}

\begin{proof}
  Immediate once $\varphi(\x)$ is replaced by some equivalent $\fins{\y}{\psi(\x,\y)}$, with $\psi$ Horn.
\end{proof}

\begin{lemma}\label{lemma: getting rid of constants from the model}
  Let $A$ be a $\Sigma$-structure, $\theory$ a theory over $\Sigma$.
  Suppose $\psi(\x)$, $\varphi(\x,\y)$ are (arbitrary first-order) formulas over $\Sigma$, and $\a \in A^\y$, such that $\varphi(\x,\a) \proves^{\theory \cup \Diag{A}}_\x \psi(\x)$.

  Then there is some regular formula $\chi(\y)$ over $\Sigma$, such that $\top \proves^{\Diag{A}} \chi(\a)$ (equivalently, $A \believes \chi(\a)$), and $\chi(\y) \land \varphi(\x,\y) \proves^\theory_{\x,\y} \psi(\x)$.
\end{lemma}

\begin{proof}
  Fix some derivation of $\varphi(\x,\a) \proves^{\theory \cup \Diag{A}}_\x \psi(\x) $.
  Let $\xi$ be the conjunction of all sentences $\sigma$ (over $\Sigma + \domain{A}$) occurring in axioms $\top \proves \sigma$ from $\Diag{A}$ in the derivation.
  Then $\xi \land \varphi(\x,\a) \proves^{\theory}_{\x}\psi(\x)$ is derivable.
  
  By \zcref{lemma: replacing constants with variables}, we can abstract this to a derivation of some sequent  $\bar{\xi}(\w) \land \bar{\varphi}(\x,\w) \proves^\theory_{\x,\w} \psi(\x)$ just over $\Sigma$, along with $f : \w \to A$ such that $\bar{\xi}(\w)[f]=\xi$ and $\bar{\varphi}(\x,\w)[f]=\varphi(\x,\a)$.
  Without loss of generality, $\w$ is disjoint from $\y$.
  Let $\rho(\y,\w)$ be the conjunction of equalities $y_i=w_j$, for $y_i$ and $w_j$ occurring in corresponding places of $\varphi$ and $\bar{\varphi}$.
  Then $\rho(\y,\w) \land \varphi(\x,\y) \proves^{\emptyset}_{\x,\y,\w} \bar{\varphi}(\x,\w)$.

  But now $\fins{\w}{(\rho(\y,\w) \land \bar{\xi}(\w))} \land \varphi(\x,\y) \proves^\theory_{\x,\y} \psi(\x)$, 
  and $f$ witnesses that $\top \proves^{\Diag{A}} \fins{\w}{(\rho(\a,\w) \land \bar{\xi}(\w))}$;
  so we are done.
\end{proof}

\section{Conservativity and completeness results for regular theories} \label{sec:conservativity}

Let $\theory$ be a regular theory over a signature $\Sigma$.
Classically, $\Mod{\theory}$ is a weakly reflective subcategory of ${\Str{\Sigma}}$ (see e.g.\ Ad\'a{}mek, Rosick\'y~\cite[Thm.~4.8, Ex.~5.e]{adamekandrosicky:94}).
Indeed, it is functorially so: every $\Sigma$-structure has a natural weakly reflective embedding $A \to W_\theory(A)$ into a model of $\theory$.
The standard proofs involve choice, and we will show in \zcref{section: reflectivity and choice principles} that this is unavoidable.

In this section, we show (constructively) a slightly weaker property, which nonetheless suffices for many applications: every $\Sigma$-structure has a natural \emph{$\theory$-conservative} embedding $\eta_A : A \to \chase[\theory]{A}$ into a model of $\theory$.

\begin{definition}\label{definition: T-conservative}
  For any $\Sigma$, $\theory$, a homomorphism of structures $f : A \to B$ is \emph{$\theory$-conservative} (w.r.t.~regular formulas) if, for every regular formula $\varphi(\x)$ and any elements $\a \in A^{l(\x)}$ such that $B \believes \varphi(f(\a))$, there is some regular formula $\psi(\x)$ such that $A \believes \psi(\a)$ and $\psi \proves^{\theory}_{\x} \varphi$.
\end{definition}

\begin{proposition}\label{proposition: T-conservative}
  Identity morphisms are $\theory$-conservative.
  Composites of $\theory$-conservative morphisms are $\theory$-conservative.
  \qed
\end{proposition}

\begin{definition}\label{definition: chase functor}
  Let $\theory$ be a regular theory over a signature $\Sigma$.
  A \emph{chase functor} for \theory\ is a  functor $\chasef{\theory} : \Str{\Sigma} \to \Mod{\theory}$ together with a natural transformation $\eta:1_{\Str{\Sigma}} \to I \cdot \chasef{\theory}$ (where $I$ is the inclusion $\Mod{\theory}\into \Str{\Sigma}$),  such that $\eta_A$ is $\theory$-conservative for every $A$.
\end{definition}

We borrow the term ``chase'' from the forward chaining algorithm of database theory of that name, as in e.g.\ Abiteboul, Hull, Vianu~\cite[\textsection 8.4]{abiteboulhullvianu:95}.
%
%
%
Our construction can be seen as an adaptation of that method, or of the categorical \emph{small-object argument} as given in e.g.\ Ad\'a{}mek et al~\cite{adamek-herrlich-rosicky-tholen}.

The basic step of the traditional chase construction is the ``axiom-induced'' extension of a structure.
Let $\sigma=(\varphi(\x) \proves_{\x}\fins{\alg{y}}\psi(\x,\y))$ be a sequent, let $M$ be a structure, and let $\alg{m}\in M$ such that $M\vDash \varphi[\alg{m}/\x]$.
Then “the extension of $M$ induced by $\sigma$ and $\alg{m}$” is the structure obtained by adding new elements $\alg{b}$ corresponding to the variables $\y$, and updating the predicates just as required to make $\psi(\a,\b)$ hold.
The map from $M$ to this extension is evidently weakly orthogonal to models of $\sigma$.

If $M$ is enumerable, or more generally well-orderable, then (classically) one can iterate this construction over all possible arguments $\alg{m}$, eventually obtaining a weak reflection $h:M\to<100> \chase[\sigma]{M}$ into a model of the sequent $\sigma$; and similarly for any well-orderable set of sequents.

In the current setting, where theories and structures need not be finite, enumerable, or even discrete, this one-at-a-time approach is insufficient.
Instead, at each step, we simultaneously adjoin new elements for all possible applications of axioms of $\theory$ to arguments in the structure (as in the small-object argument).
We then iterate this step $\omega$-many times, to obtain a model of $\theory$.

The tricky part is proving conservativity.
For this, we perform the construction directly just in the case where $\Sigma$ is purely relational (i.e.\ has no constants or function symbols), and $\theory$ does not mention equality.
This renders analysis of the basic extension step more tractable.
We then obtain the general version from this restricted case, by means of elimination of equality and function symbols in the syntax.

\begin{definition} \label{def:one-step}
  Let $\Sigma$ be a purely relational signature, and $\theory$ a regular normal equality-free theory over $\Sigma$.

  For a $\Sigma$-structure $A$, the \emph{one-step $\theory$-extension} of $A$, denoted $\onestep[\theory]{A}$ or just $\onestep{A}$, is the $\Sigma$-structure defined as follows:
  \begin{enumerate}
  \item $\domain{\onestep{A}}$ is the disjoint union of $\domain{A}$ with the set of all triples $(\tau,\a,j)$, where $\tau = (\varphi \proves_x \fins \y \psi)$ is an axiom of $\theory$, $\a \in \interp[A]{ \x\ |\ \varphi }$, and $0 \leq j < l(\y)$.
  Write $\iota_A : \domain{A} \to \domain{\onestep{A}}$ for the inclusion.

  \item For each predicate symbol $R$, $\interp[\onestep{A}]{R}$ consists of $\iota_A(\interp[A]{R})$, together with for each axiom $\tau = (\varphi \proves_{\x} \fins \y \psi)$, each $\a \in \interp[A]{ \x\ |\ \varphi }$, and each occurrence of $R$ as a conjunct $R(\t(\x,\y))$ in $\psi$, the tuple $\interp[\onestep{A}]{\x,\y\ldotp\t}(\iota(\a),\b)$, where $\b$ is the \emph{canonical witnessing tuple} $((\tau,\a,i)\ |\ 0 \leq i < l(\y)) \in \onestep{A}^{\x,\y}$.
 \end{enumerate}
 
 Moreover, this is evidently functorial in $A$; and $\iota_A$ gives a natural homomorphism $A \to \onestep{A}$.
\end{definition}

Categorically, $\onestep{A}$ is the pushout of $A$ with a coproduct of copies of the structure inclusions $\syntstr{\x}{\varphi} \to \syntstr{\x,\y}{\psi}$ representing the axioms of $\theory$ (cf.\ $C(K)$ in Ad\'a{}mek et al~\cite[II.4]{adamek-herrlich-rosicky-tholen}). 
Equality-freeness of $\theory$ ensures that this pushout can be presented simply as a disjoint union, with no quotienting required.

Some terminology will be useful for working explicitly in $\onestep{A}$.
Given a newly adjoined element $(\tau,\a,j)$, call $\tau$ the \emph{justification} for this element, $\a$ its \emph{arguments}, and $j$ its \emph{index}.

Note that whenever an instance of a basic predicate holds in $\onestep{A}$, either it already holds in $A$, or else all new elements occurring in it have the same justification $\tau$ and arguments $\a$, and the instance comes from the conclusion of $\tau$ applied to $\a$.

\begin{definition}
  ($\Sigma$ relational, $\theory$ regular normal equality-free.)

  For a $\Sigma$-structure $A$, the \emph{$\theory$-chase of $A$}, written $\chase[\theory]{A}$, is the colimit of the sequence of structures
  \[ A \to^\iota \onestep{A} \to^\iota \onestep{\onestep{A}} \to^\iota \cdots \]

  Concretely, since each $\iota$ is a complemented inclusion, $\chase{A}$ may be taken to consist of pairs $(i,x)$, where $i \in \N$, and $x \in \severalsteps[i]{A} \setminus \severalsteps[i-1]{A}$ for $i>0$, or $x \in A$ in case $i = 0$.
  Then $\chase{A} \believes R((i_1,x_1),\ldots,(i_r,x_r))$ just when there is some $j \geq i_1, \ldots, i_t$ such that $\severalsteps[j]{A} \believes R(\iota^{j - i_1}(x_1), \ldots, \iota^{j - i_r}(x_r))$.

  This too is functorial in $A$, and the colimit inclusions $\nu_i : \severalsteps[i]{A} \to \chase{A}$ are natural.
  We write $\eta_A : A \to \chase{A}$ for the $0$th such inclusion $\nu_0$.
\end{definition}

\begin{proposition}
  ($\Sigma$ relational, $\theory$ regular normal equality-free.)
  $\chase{A}$ is a model of $\theory$.
\end{proposition}

\begin{proof}
  For each axiom $\tau = (\varphi(\x) \proves_{\x} \fins \y \psi(\x,\y))$ of $\theory$, suppose $\a \in \chase{A}^{l(\x)}$, and $\chase{A} \believes \varphi(\a)$.
  Then there is some $n$ and some $\b \in \severalsteps[n]{A}^{l(\x)}$ such that $\nu_n(\b) = \a$, and $\severalsteps[n]{A} \believes \varphi(\b)$.
  Now by construction, $\severalsteps[n+1]{A} \believes \psi(\iota(\b),(\tau,\b,i)_{1 \leq i \leq l(\y)})$; so $\chase{A} \believes \psi(\a,\nu_{n+1}(\tau,\b,i)_{1 \leq i \leq l(\y)})$, and so $\chase{A}$ validates the conclusion of $\tau$.
\end{proof}

Recall \zcref{definition: T-conservative} of $\theory$-conservativity for homomorphisms.
We will show in \zcref{prop:iota-T-conservative} that $\iota_A$ is $\theory$-conservative; first we establish a restricted special case.

\begin{lemma}\label{lemma: easycase}
  ($\Sigma$ relational, $\theory$ regular normal equality-free.)  Suppose:
  \begin{enumerate}
  \item $\varphi(\x,\y)$ is some non-empty conjunction of atomic formulas, each containing some variable from $\y$;
  \item $\a \in A^{\x}$, $\b \in \onestep{A}^{\y}$;
  \item all elements of $\b$ are in $\onestep{A} \setminus A$, and have the same justification and argument;
  \item $\onestep{A} \believes \varphi(\iota(\a),\b)$.
  \end{enumerate}
  Then there is some regular formula $\psi(\x)$ such that $A \believes \psi(\a)$ and $\psi(\x) \proves_{\x}^{\theory} \fins{ \y} \varphi(\x,\y)$.
\end{lemma}

\begin{proof}
Take $\tau = (\chi \proves_{\z} \fins{\w} \zeta)$ to be the shared justification of the elements of $\b$, and $\cc \in A^{\z}$ their shared argument.
Write $\idx(b)$ for the index of $b$.

Write $\dd$ for their canonical witness tuple $(\tau,\cc,j)_{j \in \w} \in \onestep{A}^{\w}$.
Define $f : \y \to \w$ by $i \mapsto \idx(b_i)$.

Consider any non-equality atomic conjunct $R(\t(\x,\y))$ of $\varphi$, where $\t \in \Tm{\x,\y}^{\arity{R}}$.
We know $R(\t(\x,\y))$ contains some variable from $\y$; so its interpretation in $\onestep{A}$ must arise from $\tau$ applied to $\cc$.
That is, there is some conjunct of $\zeta(\z,\w)$ of the form $R(\s(\z,\w))$, where $\s \in \Tm{\z,\w}^{\arity{R}}$, such that $\interp[\onestep{A}]{\s(\z,\w)}(\iota(\cc),\dd) = \interp[\onestep{A}]{\t(\x,\y)}(\iota(\a),\b)$.

Now certainly $\zeta(\z,\w) \land \conj_{i \in \arity{R}} t_i(\z,\w) = s_i(\x,\y) \proves_{\x,\y,\z,\w} R(\t(\x,\y))$.

For each $i\in \arity{R}$, we have $\interp[\onestep{A}]{t_i(\x,\y)}(\iota(\a),\b) = \interp[\onestep{A}]{s_i(\z,\w)}(\iota(\cc),\dd)$.
Since each term is a variable, and the images of $\b,\dd$ are distinct from $\iota(\a),\iota(\cc)$, either $t_i(\x,\y)$ is of the form $x_j$ while $s_i(\z,\w)$ is of the form $z_k$, or else $t_i(\x,\y)$ is of the form $y_j$ while $s_i(\z,\w)$ is exactly $w_{f(j)}$.

So let $\epsilon(\x,\z)$ be the conjunction, over all $i$ such that $t_i(\x,\y)$ is some $x_j$, of the equalities $t_i(\x) = s_i(\z)$.
Then $\zeta(\z,\w) \land \epsilon(\x,\z)\land \conj_{1\leq i\leq l(\y)} y_i = w_{f(i)} \proves_{\x,\y,\z,\w} R(\t(\x,\y))$; and since $\iota$ is injective, $A \believes \epsilon(\a,\cc)$.

Take $\rho(\x,\z)$ to be the conjunction of these formulas $\epsilon(\x,\z)$ over all non-equality atomic conjuncts of $\varphi$.
Then 
\[ \rho(\x,\z) \land \zeta(\z,\w) \land \conj_{1\leq i\leq l(\y)} y_i = w_{f(i)} \proves_{\x,\y,\z,\w} \varphi(\x,\y): \]
the left-hand side implies all non-equality conjuncts of $\varphi$ by the construction of $\rho$, and all equality conjuncts since if $b_i = b_j$ then $f(i) = f(j)$.
Meanwhile, $\onestep{A} \believes \rho(\iota(\a),\iota(\cc))$, so since $\iota$ is an injection and $\rho$ a conjunction of equalities, $A \believes \rho(\a,\cc)$.

Finally, take the desired formula $\psi(\x)$ to be $\fins{ \z} \rho(\x,\z) \land \chi(\z)$.
Then $A \believes \psi(\a)$, witnessed by $\cc$, and moreover
\opt{arxiv}{\begin{align*}
  \psi(\x) & \proves_{\x}^{\theory} \exists z,\w\ \rho(\x,\z) \land \zeta(\z,\w) \\
  & \proves_{\x}^{\theory} \textstyle \exists \z,\w,\y\  \rho(\x,\z) \land \zeta(\z,\w) \land \conj_{i} y_i = w_{f(i)}\\
  & \proves_{\x}^{\theory} \exists \z,\w,\y \ \varphi(\x,\y) \\
  & \proves_{\x}^{\theory} \exists \y\ \varphi(\x,\y). \qedhere
\end{align*}}%
\opt{jsl}{%
\begin{endproofeqnarray*}
  \psi(\x) & \proves_{\x}^{\theory} & \exists z,\w\ \rho(\x,\z) \land \zeta(\z,\w) \\
  & \proves_{\x}^{\theory} & \textstyle \exists \z,\w,\y\  \rho(\x,\z) \land \zeta(\z,\w) \land \conj_{i} y_i = w_{f(i)}\\
  & \proves_{\x}^{\theory} & \exists \z,\w,\y \ \varphi(\x,\y) \\
  & \proves_{\x}^{\theory} & \exists \y\ \varphi(\x,\y).
\end{endproofeqnarray*}
}%
\end{proof}

\begin{proposition} \label{prop:iota-T-conservative}
  ($\Sigma$ relational, $\theory$ regular normal equality-free.)
  The map $\iota : A \to \onestep{A}$ is $\theory$-conservative, for every $\Sigma$-structure $A$.
\end{proposition}

\begin{proof}
Let $\varphi(\x)$ be a regular formula, and $\a \in A^{l(\x)}$ be such that $\onestep{A} \believes \varphi(\iota(\a))$.
We seek some $\psi(\x)$ such that $A \believes \psi(\a)$ and $\psi(\x) \proves_{\x}^{\theory} \varphi(\x)$.

As usual, $\varphi$ is provably equivalent to some formula $\fins{\y} \varphi'$, with $\varphi'(\x,\y)$ a finite conjunction of atomic formulas.
By adjoining instances of reflexivity if necessary, we may assume that each variable of $\y$ occurs somewhere in $\varphi'$.
Now take some $\b \in \onestep{A}^{l(\y)}$ such that $\onestep{A} \believes \varphi'(\iota(\a),\b)$.

Since $\iota : A \to \onestep{A}$ is a decidable injection, we can decide which values of $\b$ are in its image, and correspondingly reorder $\y$ into the form $\x',\y'$, such that the corresponding reordering of $\b$ is of the form $\iota(\a'),\b'$, with $\a' \in A^{l(\x')}$, and $\b'$ lying entirely in the complement of $\iota$.

Now let $\sim$ be the equivalence relation on $\y'$ generated by setting $y'_i \sim y'_j$ whenever there is some atomic conjunct of $\varphi'$ in which both $y'_i$ and $y'_j$ occur.
This is decidable, so allows us to reorder $\y'$ as $\y^1,\ldots,\y^n$, where each $\y^i$ is an equivalence class of $\sim$.
Write $\b^1,\ldots,\b^n$ for the corresponding reordering of $\b'$.
By the definition of $\sim$ and the note following \zcref{def:one-step}, for each $i$, all the elements in $\b^i$ must have the same justification and argument.

By reordering quantifiers and conjuncts in $\fins{\y'} \varphi'(\x,\y)$, we now have $\varphi(\x)$ equivalent to some formula
\[ \fins{\x'} \fins{\y^1} \cdots \fins{\y^n} \varphi_0(\x,\x') \land \varphi_1(\x,\x',\y^1) \land \cdots \land  \varphi_n(\x,\x',\y^n)\]
where each $\varphi_i$ is a non-empty conjunction of atomic formulas, and for $i > 0$, each conjunct of $\varphi_i$ containing some variable from $\y^i$.

This is in turn equivalent to
\[ \fins{\x'} \left( \varphi_0(\x,\x') \land (\fins{\y^1} \varphi_1(\x,\x',\y^1)) \land \cdots \land (\fins{\y^n} \varphi_n(\x,\x',\y^n) \right)). \]
To unify the form of these conjuncts, take $\y^0$ to be the empty sequence of variables.

It is now sufficient to show that for each $0 \leq i \leq n$, there is some formula $\psi_i(\x,\x')$ such that $A \believes \psi_i(\a,\a')$, and $\psi_i \proves^{\theory}_{\x,\x'}  \fins{\y^i} \varphi_i$.
Given these, choosing some such $\psi_i$ for each $i$ and taking $\psi(\x) := \fins{\x'} \conj_i \psi_i(\x,\x')$ completes the proof.

For $i \neq 0$, this is exactly \zcref{lemma: easycase}.

In the case $i = 0$, $\y^0$ is empty and so $\fins{\y^0} \varphi_0(\x,\x',\y^0)$ is just $\varphi_0(\x,\x')$.
Here it is enough to show that for each atomic conjunct $\alpha(\x,\x')$ of $\varphi_0$, there is some $\sigma(\x,\x')$ with $A \believes \sigma(\a,\a')$ and $\sigma(\x,\x') \proves^{\theory}_{\x,\x'} \alpha(\x,\x')$.
Given this, we are done by choosing some such $\sigma$ for each $\alpha$ and taking $\psi_0$ to be their conjunction.

So, take some such $\alpha(\x,\x')$.
By the explicit description of the structure on $\onestep{A}$, the fact that $\onestep{A} \believes \alpha(\iota(\a),\iota(\a'))$ must arise either because $A \believes \alpha(\a,\a')$, or else from some axiom $\tau$, some arguments $\cc$, and some atomic conjunct of the conclusion of $\tau$ applied to $\cc$.
In the former case, we are done just by taking $\sigma$ to be $\alpha$ itself.

In the latter case, say $\tau$ is of the form $\chi \proves_{\z} \fins{\w} \zeta$, and $R(\s(\z,\w))$ the atomic conjunct of $\zeta$ giving rise to the fact that $\onestep{A} \believes \alpha(\iota(\a),\iota(\a'))$.

Now $\alpha(\x,\x')$ must itself be of the form $R(\t(\x,\x'))$, \IfFileExists{newtxmath.sty}{and with}{where} $\t(\iota(\a),\iota(\a')) = \s(\iota(\cc),(\tau,\cc,j)_{1 \leq j \leq l(\w)}),$ in $\onestep{A}$; so $\s$ must be just of the form $\s(\z)$, not mentioning $\w$, and since $\iota$ is injective, $\t(\a,\a') = \s(\cc)$ in $A$.
So taking $\rho(\x,\x',\z)$ to be the conjunction of equalities $\conj_{1 \leq j \leq \arity(R)} t_j(\x,\x') = s_j(\z)$, and taking $\sigma(\x,\x')$ to be $\fins{\z} \left( \rho(\x,\x',\z) \land R(\s(\z))\right)$, we are done.
\end{proof}

\begin{corollary} \label{cor:equality-free-chase}
  ($\Sigma$ relational, $\theory$ regular normal equality-free.)  The embedding $\eta_A : A \to \chase{A}$ is \theory-conservative, for every $\Sigma$-structure $A$; so $\chasef{\theory}$ and $\eta$ form a chase functor for $\theory$.
\end{corollary}

\begin{proof}
Since $\theory$-conservative maps are closed under composition and identities, $\iota^n : A \to \severalsteps[n]{A}$ is $\theory$-conservative, for all $n \geq 0$.
But given $\varphi(\x)$ and $\a$ such that $\chase{A} \believes \varphi(\eta(\a))$, by finitariness there is some $n$ such that $\severalsteps[n]{A} \believes \varphi(\iota^n(\a))$, at which point we are done by conservativity of $\iota^n$.
\end{proof}

We can now ease our restrictions on $\theory$ and $\Sigma$, using first \emph{elimination of equalities} to remove “$\theory$ equality-free”, and then \emph{elimination of function symbols} to remove “$\Sigma$ purely relational”.
These techniques are standard, but formulations of them vary widely, so we recall carefully the specific versions that we require.

\begin{definition}[Elimination of equalities; cf.\ Bell, Machover~{\cite[\textsection 2.11]{bell-machover:1977}}] \label{def:elimination-of-equalities}
  Let $\Sigma$ be any signature; take $\Sigma_E$ to be the extension of $\Sigma$ by a new binary predicate symbol $E$,
  and let $\E{\Sigma}$ be the theory stating that $E$ is an equivalence relation and all other predicate symbols respect $E$.

  For a formula $\varphi$ over $\Sigma$, let $\varphi^E$ be the formula over $\Sigma_E$ given by replacing equality in $\varphi$ with $E$; conversely, for a formula $\psi$ over $\Sigma_E$, let $\backE{\psi}$ be the formula over $\Sigma$ given by replacing $E$ in $\psi$ with equality.
  These translations extend to sequents and theories in the obvious way.
\end{definition}

\begin{proposition}
  Both translations preserve provability, modulo $\E{\Sigma}$.
  That is, if $\varphi_1 \proves^{\theory}_\x \varphi_2$ (over $\Sigma$), then  $\varphi_1^E \proves^{\theory^E \cup \E{\Sigma}}_\x \varphi_2^E$;
  and if $\psi_1 \proves^{\theory \cup \E{\Sigma}}_\x \psi_2$ (over $\Sigma_E$), then  $\backE{\psi_1} \proves^{\backE{\theory}}_\x \backE{\psi_2}$.
  Also, $\backE{(\varphi^E)} = \varphi$, for any formula $\varphi$ over $\Sigma$. \qed
\end{proposition}

\begin{proposition}
  The composite forgetful functor $u : \Mod{\E{\Sigma}} \to \Str{\Sigma_E} \to \Str{\Sigma}$ has a left adjoint $e : \Str{\Sigma}\to \Mod{\E{\Sigma}}$, which extends any $\Sigma$-structure to a $\Sigma_E$-structure by interpreting $E$ as equality; and $e$ has a further left adjoint $q : \Mod{\E{\Sigma}}\to \Str{\Sigma}$, which quotients a structure by $E$.
  Moreover, the adjunction $q \adjoint e$ is a reflection: its counit $q(e(A)) \to A$ is a natural isomorphism. 

  Let $\varphi(\x)$ be any formula over $\Sigma$, $A \in \Str{\Sigma}$, and $\a \in A^\x$.
  Then $A \believes \varphi(\a)$ if and only if $e(A) \believes \varphi^E(\a)$.

  Let $\psi(\x)$ be any formula over $\Sigma_E$, $B \in \Str{\Sigma_E}$, and $\b \in B^\x$.
  If $B \believes \psi(\b)$, then $q(B) \believes \backE{\psi([\b])}$.

  Finally, for any theory $\theory$ over $\Sigma$, the adjunction $q \adjoint e$ restricts to an adjunction between the subcategories $\Mod{\theory}$ and $\Mod{\theory^E \cup \E{\Sigma}}$; and $q : \Mod{\E{\Sigma}}\to \Str{\Sigma}$ sends $\theory^E$-conservative maps to $\theory$-conservative maps. \qed
\end{proposition}

\begin{proposition} \label{prop:chase-with-equality}
  Any regular theory $\theory$ (possibly involving equality) over a purely relational signature $\Sigma$ has a chase functor.
\end{proposition}

\begin{proof}
  By \zcref{cor:equality-free-chase}, $\theory^E \cup \E{\Sigma}$ has a chase functor, which we write as $\chasef{E}$, $\eta^E$.
  
  Consider these as restricted to the subcategory $\Mod{\E{\Sigma}}$.
  Then take $\chasef{\theory}$ to be the composite $q\cdot \chasef{E} \cdot e : \Str{\Sigma} \to \Mod{\E{\Sigma}} \to \Mod{\theory^E \cup \E{\Sigma}} \to \Mod{\theory}$; and take $\eta^{\theory}_A$ to be the composite of $q(\eta^{E}_{e(A)}) : q(e(A)) \to q(\chase[\theory^E]{e(A)})$ with the natural isomorphism $A \iso q(e(A))$.
\end{proof}
(Note that when $\theory$ is already equality-free, \chasef{\theory} provided by this lemma is naturally isomorphic to the previous version.)

\begin{proposition}[Elimination of function symbols; cf.\ Bell, Machover~{\cite[\textsection 2.10]{bell-machover:1977}}] \label{prop:elimination-of-functions}
  Let $\Sigma$ be any signature.
  Let $\bar{\Sigma}$ be the signature obtained from $\Sigma$ by replacing each $n$-ary function symbol by an $(n+1)$-ary predicate symbol, and let $\F{\Sigma}$ the regular theory over $\bar{\Sigma}$ asserting that the new predicate symbols are functional.
  Then there is a translation $\bar{\cdot} : \Form[\Sigma]{-} \to \Form[\bar{\Sigma}]{-}$, preserving the regular fragment of logic, and conservative modulo $\F{\Sigma}$, i.e.\ such that $\varphi_1, \ldots, \varphi_n \proves_\x \psi $ if and only if $\bar{\varphi_1}, \ldots, \bar{\varphi_n} \proves^{\F{\Sigma}}_\x \bar{\psi} $, for all suitable $\x, \varphi_i, \psi$. 
  Moreover, this translation induces an equivalence of categories $\Str{\Sigma} \equiv \Mod{\F{\Sigma}}$. \qed
\end{proposition}

\begin{theorem}\label{theorem: main chase theorem}
  Let $\Sigma$ be an arbitrary signature, and \theory\ a regular theory over $\Sigma$.
  Then there exists a chase functor $\chasef{\theory} : \Str{\Sigma} \to \Mod{\theory}$ for \theory.
\end{theorem}

\begin{proof}
  Let $\bar{\Sigma}$ and the translation $\bar{\cdot}$ be as in \zcref{prop:elimination-of-functions}.
  Take $\bar{\theory}$ to be the normalisation of the regular theory
  \[ \F{\Sigma} \cup \{ \bar\varphi \proves_\x \bar\psi\ |\ (\varphi \proves_\x \psi) \in \theory \}. \]

  The equivalence $\Str{\Sigma} \equiv \Mod{\F{\Sigma}}$ restricts to an equivalence $\Mod{\theory} \equiv \Mod{\bar{\theory}}$,
  so the chase functor $\chasef{\bar{\theory}}$ provided by \zcref{prop:chase-with-equality} restricts to $\Mod{\F{\Sigma}}$ and transfers along these equivalences, yielding a chase functor for $\theory$.
\end{proof}

  (Note that in case $\Sigma$ was already purely relational and $\theory$ normal, we once again have not changed $\chasef{\theory}$, since in this case $\bar{\Sigma} = \Sigma$, $\bar{\theory} = \theory$, and the translations and equivalences involved are identities.)

We end this section by drawing, as corollaries, the regular and geometric completeness of regular theories.

\begin{corollary}\label{corollary: completeness}
  Let \theory\ be a regular theory over a signature $\Sigma$, and $\sigma$ a regular sequent valid in all models of $\theory$.
  Then $\theory \proves \sigma$.
\end{corollary}

\begin{proof}
  We assume without loss of generality that $\Sigma$ is relational (by \zcref{prop:elimination-of-functions}), and that $\sigma$ is of the form $\varphi(\x) \proves_{\x} \psi(\x)$, with $\varphi$ Horn.
  
  Write $S$ for the representing structure $\syntstr{\x}{\varphi}$, and $[\x]$ for its canonical tuple $[x_0],\ldots,[x_{n-1}] \in S^\x$.
  Now consider $\eta_S : S \to \chase[\theory]{S}$.
  Certainly $S \believes \varphi([\x])$, so since $\varphi$ is Horn, $\chase[\theory]{S} \believes \varphi(\eta_S([\x]))$.
  By hypothesis, $\chase[\theory]{S} \believes \sigma$, so $\chase[\theory]{S} \believes \psi(\eta_S([\x]))$.
  So by $\theory$-conservativity of $\eta_S$, there is some regular $\theta(\x)$ such that $S \believes \theta([\x])$ and $\theta(\x) \proves_\x^\theory \psi(\x)$.

  But by the characterisation of validity in $S$, this implies that $\varphi(\x) \proves^{\emptyset}_\x \theta([\x])$, and hence $\varphi(\x) \proves_\x^\theory \psi(\x)$.
\end{proof}

In fact, in this proof, the consequent of $\sigma$ could have been an arbitrary disjunction $\bigvee_{i \in I} {\psi_i}$ of regular formulas: $ \chase[\theory]{S} \believes \bigvee_{i \in I} \psi_i(\eta_S([\x]))$ means that there exists some $i\in I$ such that $\chase[\theory]{S}\believes \psi_i(\eta_S([\x]))$, whence we proceed as before.
We thus obtain a both constructive and  (Tarski-)semantical proof of the disjunction property (cf.\ Johnstone~\cite[D3.3.11]{elephant1}) and geometric completeness for regular theories:

\begin{scholium}\label{scholium: disjunction property}
  Let \theory\ be a regular theory and $\sigma$ a sequent of the form $\varphi\proves_{\x}\bigvee_I {\psi_i}$, where $\varphi$ and $\psi_i$ are regular formulas.
  If $\theory\proves \sigma$, then there is some $i\in I$ such that $\varphi\proves^{\theory}_{\x} {\psi_i}$.
  \qed
\end{scholium}

\begin{scholium}
  Let \theory\ be a regular theory over a signature $\Sigma$.
  Then any geometric sequent valid in all models of \theory\ is provable from \theory.
\end{scholium}

\begin{proof}
  For sequents of the form $\varphi\proves_{\x}\bigvee_I \fins{\y_i}{\psi_i}$, with $\varphi$ and all $\psi_i$ Horn, this is a minor adjustment of the proof of \zcref{corollary: completeness}, as sketched above.

  By Johnstone~\cite[D1.3.8]{elephant1}, any geometric sequent is provably equivalent to one of the form $\bigvee_J \fins{\z_j}{\varphi_j}\proves_{\x}\bigvee_I \fins{\y_i}{\psi_i}$, with all $\varphi_j$, $\psi_i$ Horn.
  But validity/provability of $\sigma$ is  equivalent to the validity/provability of all the sequents ${\varphi_j}\proves_{\x,\z_j}\bigvee_I \fins{\y_i}{\psi_i}$; so this reduces to the preceding case.
\end{proof}

\section{Reflectivity and choice principles}\label{section: reflectivity and choice principles}
 
In a classical meta-theory, the construction of $\chase[\theory]{A}$ shows that the category of models for a regular theory is weakly reflective in the category of structures (Ad\'a{}mek et al~\cite[II.10]{adamek-herrlich-rosicky-tholen}).
That is to say, that any homomorphism from a structure $A$ to a \theory-model $M$ factors (not necessarily uniquely) through  $\eta_A:A\to\chase[\theory]{A}$. 
\[\bfig
\ptriangle/>`->`<--/<600,350>[A`M`{\chase[\theory]{A}};`\eta_A`h]
\efig\]

The classical proof of this uses the axiom of choice; it turns out that this is unavoidable.
We show in this section that various reflectivity properties for regular theories are in fact equivalent to choice principles in the ambient set theory.

\subsection*{Background}

We start by recalling some further background in category theory and constructive choice principles.

\begin{definition}
  Let $\C$ be a category, with an initial object $0$ and terminal object $1$.

  \begin{enumerate}
  \item For maps $m$, $p$, say $m$ is \emph{weakly (left) orthogonal} to $p$ if for every commutative square from $m$ to $p$, there exists some diagonal filler.
    
\[\bfig
\square<350,350>[A`C`B`D; `m`p`]
\morphism(0,0)/-->/<350,350>[B`C;]
\efig\]
  
  \item A family of maps with common domain, $m_i : A \to B_i$ ($i \in I$), is \emph{jointly weakly (left) orthogonal} to $p : C \to D$ if for every $h : A \to C$ and family $k_i : B_i \to D$ such that $k_i m_i = ph$, there is some $i \in I$ for which there is a diagonal filler $e : B_i \to C$.

  \item A map (resp.\ family) is \emph{weakly} (resp.\ {jointly weakly}) \emph{orthogonal} to an object $C$ if it is (jointly) weakly orthogonal to the unique map $!_C : C \to 1$.%
    \footnote{Viewed from the right-hand side, these are exactly injectivity of an object with respect to a map or cone, in the sense of Ad\'a{}mek, Rosick\'{y}~\cite[4.1, 4.14]{adamekandrosicky:94}.}

  \item A family of objects $A_i$ is \emph{jointly weakly initial} if for every object $X$, for some $i$ there is a map $A_i \to X$.
    (Equivalently, if the maps $0 \to A_i$ are jointly weakly orthogonal to all objects.)

   \item  An object $X$ is \emph{projective} if the map $0 \to X$ is weakly orthogonal to all epimorphisms.
  \end{enumerate}
\end{definition}

We recall various choice principles, and the relationships between them.

\begin{definition} \needspace{3\baselineskip}  \leavevmode 
  \begin{enumerate}
  \item The \emph{axiom of choice} (\AC) states just that all sets are projective.

  \item The \emph{presentation axiom} (\PAx) (cf.~Aczel~\cite{aczel:77}, Rathjen~\cite[\textsection 4]{Rathjen02choiceprinciples}), also known as \coshep\ (“category of sets has enough projectives”), states that for every set $X$, there is some cover (i.e.\ surjection) $s : Y \surjto X$, with $Y$ projective.

  \item The axiom \WIC\ (\emph{“weakly initial cover”}) states that for every set $X$ there is some cover $s : Y \surjto X$, weakly initial in the category of covers of $X$. 

  \item The axiom \WISC\ (\emph{“weakly initial set of covers”}, studied by Van den Berg and Moerdijk in \cite{vandenBerg2014-VANTAO-18}, there called the (weak) axiom of multiple choice) states that for every set $X$, there is some family of covers $(s_i : Y_i \surjto X)_{i \in I}$, jointly weakly initial in the category of covers of $X$.
  \end{enumerate}
\end{definition}

(To our knowledge, \WIC\ has not been previously considered in the literature.)

\begin{definition} \needspace{3\baselineskip} \leavevmode
  \begin{enumerate}
  \item \emph{Dependent choice} (DC) is the statement that for any set $A$ and entire set-relation $R \subseteq A \times A$, for any $x_0\in A$ there is some function $f : \N \to A$ such that $f(0) = x_0$ and for all $n \in \N$, $(x_n,x_{n+1}) \in R$.   

  \item \emph{Relativised dependent choice} (RDC), which might also be called \emph{class dependent choice}, is the scheme asserting for all formulas $\alpha(x)$, $\rho(x,y)$ (possibly with further free variables) that if $\rho$ is an entire relation on $\alpha$, then for any $x_0$ such that $\alpha(x_0)$ there is some function $f$ on $\N$ such that $f(0) = x_0$ and for all $n \in \N$, $\alpha(x_n)$ and $\rho(x_n,x_{n+1})$. 
  
  \item The \emph{relation reflection scheme} (RRS) (Aczel \cite{aczelMALQ08}) is the scheme asserting for each pair of formulas $\alpha(x)$, $\rho(x,y)$ (possibly with further free variables) that if $\rho$ is an entire relation on $\alpha$, then for any set $A$ that is a subclass of $\alpha$, there is some set $B$, also a subclass of $\alpha$, such that $A\subseteq B$ and the restriction of $\rho$ to $B$ is still entire.
  \end{enumerate}
\end{definition}

We recall known relationships between these, and add $\WIC$ to the mix.

\begin{proposition} \label{prop:choice-principle-implications}
 \needspace{3\baselineskip}  \leavevmode 
 \begin{enumerate}
 \item $\AC \Imp \PAx \Imp \WIC \Imp \WISC$. 
 \item $\RDC \Equiv \DC + \RRS$.
 \item $\WIC \Imp \DC$.  
 \end{enumerate}
  
\end{proposition}

\begin{proof}
  $\AC \Imp \PAx \Imp \WISC$ is noted by Rathjen, \cite[\textsection \textsection 5,~6]{Rathjen02choiceprinciples}.
  The interpolation of $\WIC$ is straightforward: any cover by a projective set is weakly initial, and the singleton of any weakly initial cover is a jointly weakly initial set.
  
   $\RDC \Equiv \DC + \RRS$ is shown by Aczel, \cite[Thm.~2.4]{aczelMALQ08}.

   For $\WIC \Imp \DC$, note that the argument for $\PAx \Imp \DC$ in Blass \cite[Thm.~6.2]{blass:79} requires only weak initiality of the cover.
\end{proof}

\begin{proposition}
  (Meta-theorem.) Even under $\ZF + \REA$, \WISC\ does not imply \WIC.
\end{proposition}

\begin{proof}
  Rathjen \cite[Cor.~6.11]{Rathjen02choiceprinciples}  shows $\ZF + \AMC + \REA \not\Imp \AC_\omega$.
  But by \zcref{prop:choice-principle-implications}, $\AMC \Imp \WISC$, and $\WIC \Imp \DC \Imp \AC_\omega$, under $\CZF$.
  So $\ZF + \REA + \WISC \not\Imp \WIC$.
\end{proof}

In fact, \WIC\ suffices for many applications of \PAx\ in the literature.
It seems likely to us that \WIC\ is strictly weaker than \PAx, but this does not seem to be obvious from existing results.

We will use DC reformulated in terms of graphs (in the category-theorist’s sense).

\begin{definition}
  A \emph{graph} $G$ consists a set $G_0$ and a function $G_1 : (G_0 \times G_0) \to \Set$.
  A \emph{class graph} (meta-definition) is a formula $\gamma_0(x)$ together with a formula $\gamma_1(x,y,f)$, possibly with further free variables (“parameters”).
  Any graph $G = (G_0,G_1)$ may be considered as a class graph, given by the formulas $x \in G_0$, $f \in G_1(x,y)$.

  A class graph $(\gamma_0(x),\gamma_1(x,y,f))$ is \emph{entire} if for each $x$ such that $\gamma_0(x)$ holds, there exist some $y$, $f$ such that $\gamma_0(y)$ and $\gamma_1(x,y,f)$.

  A \emph{branch}%
  \footnote{\emph{Path} and \emph{trail} are more established for this; but they are also often used just for \emph{finite} paths, so to avoid ambiguity we borrow \emph{branch} from the special case of trees.}
in a class graph $(\gamma_0(x),\gamma_1(x,y,f))$ is a function $b$ from $\N$ to pairs $(x_n,f_n)$ such that for each $n$, both $\gamma_0(x_n)$ and $\gamma_1(x_n,x_{n+1},f_n)$ holds.
  For a (set) graph $G$, write $\Br{G}{x}$ for the set of branches starting from $x \in G_0$.
\end{definition}

\begin{proposition}
  (Over CZF.)
  DC is equivalent to the statement: for every entire graph $(G_0,G_1)$ and $x_0 \in G_0$, there is some branch in $(G_0,G_1)$ starting from $x_0$.
  RDC is equivalent to the scheme asserting for each class graph $(\gamma_0(x),\gamma_1(x,y,f))$ that for all values of the parameters, if $(\gamma_0(x),\gamma_1(x,y,f))$ is entire, then for all $x_0$ such that $\gamma_0(x_0)$ holds, there is some branch starting from $x_0$. \qed
\end{proposition}

\subsection*{Reflectivity principles for regular theories}

With this background recalled, we are ready to start on the meat of this section: showing that reflectivity principles for regular theories are equivalent to choice principles.
The main work is in building suitable weak reflections.
For this, we will proceed along similar lines to \zcref{sec:conservativity}  above: we will give generalisations of the chase construction, first in the purely relational equality-free case, and then deduce from this the general case.

For the next few statements, fix a theory $\theory$ over a signature $\Sigma$, and work in the category of $\Sigma$-structures.

\begin{definition} \needspace{3\baselineskip}  \leavevmode 
  \begin{enumerate}
  \item A map $f : A \to B$ is \emph{weakly $\theory$-reflective} if it is weakly orthogonal to all models of $\theory$.
    \[\bfig 
    \ptriangle/->`->`<--/<400,300>[A`M `B ;`f`]
    \efig\]

  \item A family of maps $A \to B_i$ is \emph{jointly weakly $\theory$-reflective} if it is jointly weakly orthogonal to all models of $\theory$.
  \end{enumerate}
\end{definition}

\begin{definition} (Reflectivity principles for theories.)
  \begin{enumerate}
  \item $\theory$ satisfies \emph{weak reflectivity} if for every $\Sigma$-structure $A$ there is some weakly $\theory$-reflective map from $A$ into a model of $\theory$.

  \item $\theory$ satisfies \emph{joint weak reflectivity} if for every $\Sigma$-structure $A$ there is some  jointly weakly $\theory$-reflective set of maps from $A$ into models of  $\theory$.

  \item $\theory$ satisfies \emph{functorial weak reflectivity} if there is a functor $F : \Str{\Sigma} \to \Mod{\theory}$ and natural transformation $\eta : 1_{\Str{\Sigma}} \to I \cdot F$ such that each $\eta_A : A \to F(A)$ is weakly $\theory$-reflective.
     (Here $I$ is the inclusion functor $\Mod{\theory} \into \Str{\Sigma}$.)
  \end{enumerate}
\end{definition}

In order to generalise the chase, we abstract the property of the one-step extension that ensures the chase yields a model of $\theory$:

\begin{definition}
    ($\theory$ regular normal.)
    A map $f : A \to B$ is \emph{$\theory$-productive} if for each axiom $\varphi(\x) \proves_{\x} \fins{\y} \psi(\x,\y)$ of $\theory$ and $\a \in \interp[A]{\x\ |\ \varphi }$, there is some $\b \in B^\y$ such that $(f(\a),\b) \in \interp[B]{\x,\y\ |\ \psi}$. 
\end{definition}

\begin{proposition} \label{prop:colim-of-t-productive}
  ($\theory$ regular normal.)
  The colimit of an $\omega$-chain of $\theory$-productive maps is a model of $\theory$.
\end{proposition}

\begin{proof}
  Immediate by finitariness of the axioms of $\theory$.
\end{proof}

\begin{proposition} \label{prop:weakly-reflective-extensions} 
  \needspace{3\baselineskip}  
  ($\Sigma$ relational, $\theory$ regular normal equality-free.)
 
  \begin{enumerate}
  \item Assuming \WISC, every $\Sigma$-structure has some jointly weakly $\theory$-reflective set of $\theory$-productive extensions.

  \item Assuming \WIC, every $\Sigma$-structure has some $\theory$-productive and weakly $\theory$-reflective extension.

  \item Assuming \AC, the functorial extension $\iota_A : A \to \onestep{A}$ is always weakly $\theory$-reflective (besides being $\theory$-productive). \label{item:eta-wk-refl}
  \end{enumerate}
\end{proposition}

\begin{proof}
  Let $A$ be a $\Sigma$-structure.
  Recall from \zcref{def:one-step} the  one-step $\theory$-extension  $\iota_A:A \to \onestep[\theory]{A}$; we now generalise this construction.

  Take $M_A$ to be the set of pairs $(\alpha,\a)$  such that $\alpha$ is an axiom of $\theory$, with the form $\varphi(\x) \proves \fins \y \psi(\x,\y)$, and $\a \in \interp[A]{ \x\ |\ \varphi(\x) }$.
  For any map $p : X \to M_A$, define a structure $A[p]$ by taking $|A[p]|$ to be the disjoint union of $A$ with the set of pairs $(x,i)$ such that $x \in X$ and $(p(x),i) \in \onestep{A}$, i.e.\ $i \in \y$, where the axiom component of $p(x)$ has the form $\varphi(\x) \proves \fins \y \psi(\x,\y)$.
  Take the structure on $A[p]$ to be induced by the evident map $|A[p]| \to |\onestep{A}|$.
  There is an evident homomorphism $\iota^p_A : A \to A[p]$, and if $p$ is surjective, then $\iota^p$ is $\theory$-productive.

  Now if $\family{p_i : M_i \to M_A}{i \in I}$ is a weakly initial set of covers, we claim that the extensions $\iota^{p_i} : A \to A[p_i]$ are jointly weakly $\theory$-reflective.
  Suppose $f : A \to B$, where $B \believes \theory$.
  Then for each $(\alpha,\a) \in M_A$, where $\alpha$ has form $\varphi(\x) \proves \fins \y \psi(\x,\y)$, there is some $\b \in B^\y$ such that $(f(\a),\b) \in \interp[B]{\y\ |\ \psi(\x,\y) }$.
  So by weak initiality, for some $i$ there is a function giving for each $m \in M_i$ some $\b$ witnessing this existential for $p(m)$.
  This induces a homomorphism $A[p_i] \to B$ extending $f$.

  In particular, if $p : M' \to M$ is a single weakly initial cover, then $\iota^p$ is weakly $\theory$-reflective.

  Finally, assuming \AC, the identity $1_{M_A}$ is itself weakly initial, so $\iota^{1_{M_A}}$ is weakly $\theory$-reflective. 
  But $\iota^{1_{M_A}} \iso \iota_A$, under $A$.
\end{proof}

The preceding proposition gives us suitably reflective one-step extensions, or families thereof.
With the aid of (R)DC, we can now take colimits of chains of these to give suitably reflective (families of) maps into models of $\theory$.

\begin{proposition} \label{prop:colim-of-reflective}
 Assume $\DC$.
 Let $\C$ be a locally small category with colimits of $\omega$-chains, and $X$ an object of $\C$.
  \begin{enumerate}
  \item For any  $\omega$-chain $F: \omega \to \C$ of maps weakly orthogonal to $X$, the transfinite composite $F(0) \to \colim_i F(i)$ is again weakly orthogonal to $X$.
    
  \item Let $G$ be a graph, and $F : G \to \C$ a diagram such that for each vertex $i \in G_0$, the set $\family{F(a) : F(i) \to F(j) }{j \in G_0,\, a \in G_1(i,j) }$ is jointly weakly orthogonal to $X$.
    Then for any $i_0 \in G_0$ the set of all maps $\{ F(i_0) \to F_b \ |\ b \in \Br{G}{i_0} \} $ is jointly weakly orthogonal to $X$.
    (Here $F_b$ denotes the colimit of $F$ along a branch $b$.)
  \end{enumerate} 
\end{proposition}

\begin{proof}
  The first statement is straightforward.
 
  For the second, consider the comma graph $F/X$, where $(F/X)_0$ consists of pairs $(i,f)$ such that $i \in G_0$, $f : F(i) \to X$, and $(F/X)_1((i,f),(j,g))$ is the set of arrows $a \in G_1(i,j)$ such that $g \cdot F(a) = f$.
  The joint reflectivity assumption on $F$ implies directly that $F/X$ is entire.

  So given any $f_0 : F(i_0) \to X$, \DC\ gives us a branch in $F/X$ starting from $(i_0,f_0)$.
  The first components of this form a branch $b$ in $G$; and the second components induce a map $F_b \to X$ extending $f_0$. 
\end{proof}

\begin{proposition} \label{prop:wic-imp-wk-refl}
  ($\Sigma$ relational, $\theory$ regular normal equality-free.)
  Assuming \WIC\ and either \RRS\ or \RDC, every structure has some weakly $\theory$-reflective map into a $\theory$-model.
\end{proposition}

\begin{proof}
  (By \zcref{prop:choice-principle-implications}, \RRS\ is equivalent to \RDC\ in the presence of \WIC.)

  Let $A$ be any $\Sigma$-structure.
  By \zcref{prop:weakly-reflective-extensions} together with \RDC, we can take some branch $A = A_0 \to A_1 \to A_2 \to \ldots$, in which each step is $\theory$-productive and weakly $\theory$-reflective.

  Now set $\bar{A} = \colim_i A_i$.
  By \zcref{prop:colim-of-t-productive}, $\bar{A} \believes \theory$; and by \zcref{prop:colim-of-reflective}, $A \to \bar{A}$ is weakly $\theory$-reflective.
\end{proof}

\begin{proposition} \label{prop:wisc-imp-joint-refl}
  ($\Sigma$ relational, $\theory$ regular normal equality-free.)
  Assuming \RDC\ and \WISC, every structure has some jointly weakly $\theory$-reflective set of maps into $\theory$-models.
\end{proposition}

\begin{proof}
  First note that given any set $\X$ of $\Sigma$-structures, there exists some set $\Y$ of $\Sigma$-structures such that for each $A \in \X$, the set of all $\theory$-productive maps from $A$ into structures in $\Y$ is jointly weakly $\theory$-reflective. 
  This follows from \zcref{prop:weakly-reflective-extensions} together with Collection, and the fact that any superset of a jointly weakly reflective set is jointly weakly reflective.

  Now let $A_0$ be any $\Sigma$-structure.
  By \RRS\ applied to $\{A_0\}$ and the entire relation of the previous paragraph, there is some set $\X$ of $\Sigma$-structures containing $A_0$ such that for each $A \in \X$, the set of all $\theory$-productive maps from $A$ into structures in $\X$ is jointly weakly $\theory$-reflective.
  
  Take $G$ to be the graph with objects $\X$, and with arrows all $\theory$-productive maps between these.
  Then \zcref{prop:colim-of-t-productive}, applied to the evident diagram $G \to \Str{\Sigma}$, tells us that the maps $\family{A_0 \to A_b}{b \in \Br{G}{A_0}}$ form a jointly $\theory$-reflective family.
  But by \zcref{prop:colim-of-t-productive}, each such $A_b$ is a model of $\theory$; so we are done.
\end{proof}

Finally, we can put all the pieces together:

\begin{theorem}[“Choice principles $\Equiv$ reflectivity principles”]  \needspace{3\baselineskip}  \leavevmode 
  \begin{enumerate}
  \item \AC\ is equivalent to “all regular theories satisfy functorial weak reflectivity”.

  \item Assuming \RRS\ (or a fortiori \RDC), \WIC\ is equivalent to “all regular theories satisfy weak reflectivity”.

  \item Assuming \RDC, \WISC\ is equivalent to “all regular theories satisfy joint weak reflectivity”.
  \end{enumerate}
\end{theorem}

\begin{proof}
  We tackle first the $\Imp$ directions of the three equivalences.
  
  For $\Sigma$ purely relational and $\theory$ equality-free, these are \zcref{prop:weakly-reflective-extensions} (\ref{item:eta-wk-refl}), \zcref{prop:wic-imp-wk-refl}, and \zcref{prop:wisc-imp-joint-refl} respectively.
  
  For $\Sigma$ purely relational but $\theory$ not necessarily equality-free, these follow from the equality-free cases by elimination of equalities, as set up in \zcref{def:elimination-of-equalities}--\zcref{prop:chase-with-equality}.
  For weak reflectivity, given $A \in \Str{\Sigma}$, consider it as a $\Sigma_E$ structure $e(A)$, take a weakly $(\theory^E \cup \E{\Sigma})$-reflective map into a $(\theory^E \cup \E{\Sigma})$-model $M$, and apply $q$ to get a map from $A$ ($\iso q(e(A))$) into the quotient $\theory$-model $q(M)$.
  The resulting map $A \to q(M)$ is weakly $\theory$-reflective, by the reflection $q \adjoint e$.
  The cases for joint and functorial reflectivity are analogous.

  Finally, the case for general $\Sigma$ and $\theory$ reduce to the relational case by elimination of function symbols, \zcref{prop:elimination-of-functions}.
  Replace $\Sigma$ by an associated signature $\bar{\Sigma}$ and theory $\F{\Sigma}$ thereover, such that $\Str{\Sigma} \equiv \Mod{\F{\Sigma}}$, and moreover this restricts to $\Mod{\theory} \equiv \Mod{\bar{\theory}}$ for some theory $\bar{\theory}$ associated to $\theory$.
  Then the (functorial, weak, joint) reflection from $\Str{\bar{\Sigma}}$ into $\Mod{\bar{\theory}}$ restricts to a similar reflection from $\Str{\Sigma}$ into $\Mod{\theory}$.

  This completes the $\Imp$ directions of each equivalence. 

  \subtleparagraph For the $\ImpliedBy$ directions, consider the signature $\Sigma$ with just one unary predicate $P$ and one binary predicate $R$, and the theory $\theory$ consisting just of the axiom $P(x) \proves_x \exists y.\ R(x,y)$.
  There are evident functors $F : \Set^\arr \to \Str{\Sigma}$ and $G : \Str{\Sigma} \to \Set^\arr$, acting on objects as follows:
  \begin{enumerate}
    \item $F$ sends a function $f : Y \to X$ to the $\Sigma$-structure on $X + Y$ with $\interp{P} = \nu_1[X]$, $\interp{R} = \{ (\nu_1(fy),\nu_2(y))\ |\ y \in Y \}$;
    \item $G$ sends a $\Sigma$-structure $A$ to the projection map $\interp[A]{ x,y\ |\ P(x) \land R(x,y) }$ $\to$ $\interp[A]{ x\ |\ P(x) }$, which we denote by $g_A : G_1(A) \to G_0(A)$.
  \end{enumerate}
  It is easy to check that $F$ is left adjoint to $G$; $F$ is full and faithful; and $F$ and $G$ preserve and reflect the subcategories of $\theory$-models and covers.

  We identify any set $X$ with the unique map $0 \to X$, as an object of $\Set^\arr$.
  In particular, by $1$ we mean for now the map $0 \to 1$, not the terminal object of $\Set^\arr$.

  We now take the three $\ImpliedBy$ directions in turn.
  
  Firstly, suppose $W : \Str{\Sigma} \to \Str{\Sigma}$ is a functorial weak reflection, with unit $\eta : 1 \to W$; we will show that \AC\ holds.
  
  Take any singleton $1 = \{\star\}$.
  Then $W(F(1))$ contains some element $\diamond$, such that $W(F(1)) \believes R(\eta(\nu_1(\star)),\diamond)$.
    
  Now, let $p : Y \surjto X$ be any cover; we will construct a splitting of $p$.
    
  There is a unique map $\iota : X \to p$ in $\Set^\arr$ over $1_X$.
  This induces a map $F \iota : F(X) \to F(p)$, which by the universal property of the weak reflector factors through $W(F(X))$ as $F \iota = f \comp \eta_{F(X)}$, for some $f$.
    
  We can now obtain a splitting of $p$, by first finding a witnessing function for the existentials in $W(F(X))$, and then composing this with $f$.
  Precisely, each $x \in X$ determines
  a map  $\hat{x} : 1 \to X$, and hence a homomorphism $\name{x} := F\hat{x} : F(1) \to F(X)$.
  For each $x \in X$, naturality of $\eta$ ensures that $W(F(X)) \believes R(\eta_{F(X)}(\nu_1(x)), W\name{x}(\diamond))$.
  So the map $X \to X + Y$ $(= \domain{F(p)})$ sending $x$ to $f(W\name{x}(\diamond))$ factors through $Y$; and this factorisation gives a section of $p$.
  \[\bfig
  \square/>`<-`<-`>/<750,500>[W(F(1))`W(F(X))`F(1)`F(X);W\name{x}`\eta_{F(1)}`\eta_{F(X)}`\name{x}]
  \ptriangle(750,0)/->`<-`<-/<750,500>[W(F(X))`F(p)`F(X);f``F\iota]
  
  \efig\]
  
  \subtleparagraph Secondly, suppose weak reflectivity holds for $\theory$; we want to show that \WIC\ holds.
  So let $X$ be any set; we will find a weakly initial cover for $X$.
  
  Consider the $\Sigma$-structure $F(X)$ as before, and let $\eta : F(X) \to A$ be some weak reflection into a $\theory$-model.
  Applying $G$ to $A$ gives a cover $g_A : G_1(A) \surjto G_0(A)$, and transposing $\eta$ under the adjunction yields a map $\hat{\eta} : X \to G_0(A)$.
  Pulling back $g_A$ along $\hat{\eta}$ gives a cover $p : \bar{X} \surjto X$, and a factorisation of $\hat{\eta}$ through the canonical inclusion $\iota : X \to p$.
  
  Transposing again, we get a factorisation of $\eta : F(X) \to A$ through $F (\iota) : F(X) \to F(p)$,
  so $F (\iota)$ is weakly $\theory$-reflective since $\eta$ was.
  Since $F$ is full and faithful and sends covers to models of $\theory$, it follows that $\iota : X \to p$ is left-orthogonal to all covers (viewed as objects of $\Set^\arr$).
  That is, $p : \bar{X} \to X$ is a weakly initial cover. 

  \subtleparagraph Finally, to show that joint reflectivity for $\theory$ implies \WISC, we perform the same construction as above on each map in a jointly reflective set $\eta_i : F(X) \to A_i$, and obtain (by essentially the same argument) a weakly initial set of covers $p_i$.
\end{proof}

\opt{arxiv}{\bibliographystyle{amsalphaurlmod}}
\opt{jsl}{\bibliographystyle{asl}}
\bibliography{non-discrete}

\end{document}